\documentclass[submission,copyright,creativecommons]{eptcs}
\providecommand{\event}{AiML 2026} 

\usepackage{iftex}

\ifpdf
  \usepackage{underscore}         
  \usepackage[T1]{fontenc}        
\else
  \usepackage{breakurl}           
\fi

\usepackage{amsmath, amssymb, amscd, mathrsfs, amsthm}
\usepackage{comment}
\usepackage{tikz}

\newtheorem{theorem}{Theorem}[section]
\newtheorem{cor}[theorem]{Corollary}

\newtheorem{prop}[theorem]{Proposition}
\newtheorem*{mytheorem}{Theorem}

\theoremstyle{definition}
\newtheorem{definition}{Definition}
\newtheorem{example}{Example}
\newtheorem{remark}{Remark}

\newcommand{\alg}[1]{{{\textbf{\upshape #1}}}}

\newcommand{\VV}{\mathbf{V}}
\newcommand{\QQ}{\mathbf{Q}}
\newcommand{\PR}{\mathbf{P}}

\newcommand{\SU}{\pmb{S}}
\newcommand{\II}{\pmb{I}}
\newcommand{\HH}{\pmb{H}}
\newcommand{\PP}{\pmb{P}}
\newcommand{\PPu}{\pmb{P}_{\!\!u}}

\newcommand{\UUk}{\pmb{U}_{\!\!\kappa}}

\newcommand{\vv}[1]{\mathsf {#1}}

\newcommand{\vV}{\vv V}
\newcommand{\vK}{\vv K}

\newcommand{\pVk}{\WsV(\kk)}

\newcommand{\ee}{\epsilon}

\newcommand{\bb}{\mathfrak{b}}

\newcommand{\FrV}{\alg F_\vv V(\omega)}

\newcommand{\HA}{\mathsf{HA}}

\newcommand{\Def}[1]{\textit{#1}}

\newcommand{\app}{\approx}

\newcommand{\set}[2]{\{#1 \mid #2\}}

\newcommand{\vt}{\vartheta}

\newcommand{\WsV}{\mathfrak{b}_\vv V}

\newcommand{\seq}[2]{#1 \Ra  {#2}}

\newcommand{\GG}{\Gamma}	
\newcommand{\gm}{\gamma}

\newcommand{\kk}{\kappa}
\newcommand{\lm}{\lambda}

\newcommand{\one}{\mathbf{1}}
\newcommand{\zero}{\mathbf{0}}

\newcommand{\Ra}{\Rightarrow}

\newcommand{\sq}{\mathfrak{s}}
\newcommand{\Sq}{\mathfrak{S}}
\newcommand{\csq}{\mathfrak{c}}

\newcommand{\Vk}{\vv V^\kk}
\newcommand{\VP}{\vv V^p}

\title{On Strong Structural Completeness of Varieties and Quasivarieties}
\author{Alex Citkin
\institute{Metropolitan Telecommunications\\ New York, USA}
\email{acitkin@gmail.com}
}

\def\titlerunning{On Strong Structural Completeness}
\def\authorrunning{Alex Citkin}

\begin{document}
\maketitle

\begin{abstract}
We study structural completeness in the infinitary sense (strong structural completeness) in an algebraic setting. A variety is structurally complete (SCpl) if it is generated, as a quasivariety, by its free algebras, and it is strongly structurally complete (SSCpl) if it is generated, as a prevariety, by its free algebras. A quasivariety is SSCpl if it is generated, as a prevariety, by its free algebras.

We prove that every quasivariety of finite type with the CEP that is generated by finite algebras and contains an infinite irreducible algebra is not SSCpl. Moreover, every congruence meet-semidistributive variety of finite type generated by finite algebras is SSCpl if and only if it is tabular. Thus, Dummett’s and Medvedev’s logics are SCpl but not SSCpl.

A variety is primitive if it is SCpl and all its subvarieties are SCpl; it is strongly primitive if it is SSCpl and all its subvarieties are SSCpl. We prove that in primitive congruence-distributive varieties of finite type, the tabular subvarieties, and only those, are strongly primitive. This observation also yields a criterion for strong primitivity.
\\


\end{abstract}

\section{Introduction}
 
Propositional logic, understood as a consequence relation, can be defined in two ways: syntactically, by a deductive system—a pair consisting of a set of axioms and a set of structural inference rules—and semantically, by a class of models that, for each valuation (assignment), validates the conclusion whenever all the premises are validated. Since derivations are finite sequences of formulas, every deductive system defines a finitary structural consequence relation, whose equivalent algebraic semantics is a quasivariety (see \cite{BlokPigozziAlgebraizable1989}). Moreover, any finitary structural consequence relation can be axiomatized by a deductive system (see \cite{LosSuszkoRemarks1958}). In addition, it was proved there that if the defining matrix is finite, then the consequence relation (or operator) determined by it is finitary.

However, in general, consequence relations determined by an infinite algebra or by a class of algebras need not be finitary: even the simplest infinite Heyting algebra of order type $\omega + 1$ determines a non-finitary consequence relation. Note that this consequence relation is also determined by the class of all finite linearly ordered Heyting algebras; thus, infinite classes of finite algebras may define a non-finitary consequence relation.

Not surprisingly, any finite set of similar finite algebras of finite type always determines a finitary consequence relation (see Section~\ref{sec-SSCpl}, where we give an algebraic proof of a generalization of the theorem from \cite[Section~8]{LosSuszkoRemarks1958}, stating that a consequence relation defined by a finite matrix is finitary).

On the other hand, in Section~\ref{sec-pr}, we prove that any infinite set of pairwise nonisomorphic finite algebras from a congruence semi-distributive variety always defines a non-finitary consequence relation. This justifies the interest in studying infinitary consequence relations.

The algebraic counterparts of finitary consequence relations are quasivarieties, that is, classes of algebras that can be defined by a set of quasiequations. A \Def{quasiequation} is a $\kk$-quasiequation of the form $\ee_1, \dots, \ee_n \Ra \ee$, where $\ee, \ee_i$ ($1 \leq i \leq n$) are equations. Thus, every quasiequation contains only finitely many (possibly none) premises and therefore only finitely many variables.

At the same time, the algebraic counterparts of non-finitary consequence relations are prevarieties (also known as implicative classes \cite{ShafaatImplicationally1969} or generalized quasivarieties \cite{MoraschiniLogical2018}; cf.\ also \cite{BlokJonsson1999}), that is, classes of algebras defined by $\kk$-quasiequations of the form $E \Ra \ee$, where $E$ is a (possibly empty or infinite) set of equations. If $E$ is infinite, there are two important parameters associated with such a $\kk$-quasiequation: $\lm$—the cardinality of $E$—and $\kk$—the cardinality of the set of variables occurring in the $\kk$-quasiequation.

For any regular cardinal $\lm > 0$ and any infinite regular cardinal $\kk$, we define a $(\lm,\kk)$-quasiequation as a $\kk$-quasiequation $E \Ra \ee$ such that $|E| < \lm$ and $E$ contains fewer than $\kk$ distinct variables (see \cite[Section~2.3.4]{GorbunovBookE}). By a $\kk$-\Def{quasiequation} we mean a $(\kk,\kk)$-quasiequation. Thus, $(1,\kk)$-quasiequations are equations, and $\omega$-quasiequations are the usual (finitary) quasiequations.

Prevarieties that can be defined by a set of $\kk$-quasiequations will be called $\kk$-\Def{quasivarieties}. By $\QQ_\kk(\vv K)$ we denote the smallest $\kk$-quasivariety containing the class of similar algebras $\vv K$. Thus, $\QQ_\omega(\vv K) = \QQ(\vv K)$ is the quasivariety generated by $\vv K$.
\smallskip

The notion of structural completeness was introduced by W.~Pogorzelski in his seminal 1971 paper \cite{PogorzelskiStructural1971}. It was defined for consequence operators determined by a deductive system $\langle R, A \rangle$, where $R$ is a finite set of (structural finitary) rules and $A$ is a set of propositional formulas. A rule $r$ is \Def{admissible} for the operator if the operator determined by the extended deductive system $\langle R \cup {r}, A \rangle$ has the same set of theorems as the original operator, and $r$ is \Def{derivable} if the deductive system $\langle R \cup {r}, A \rangle$ determines the same operator. The operator is said to be \Def{structurally complete} if every rule admissible for it is derivable. In the algebraic setting, a variety $\vv V$ is structurally complete if it is generated as a quasivariety by its free algebras, that is, if $\vv V = \QQ(\FrV)$.

We say that a variety $\vv V$ is $\kk$-\Def{structurally complete} if $\vv V = \QQ_\kk(\FrV)$ (in Section~\ref{sec-pr} we justify why it is sufficient to consider free algebras of countable rank; cf.\ Theorem~\ref{th-freesame}). Obviously, $\omega$-structural completeness coincides with ordinary structural completeness.
\smallskip

The goal of this paper is to demonstrate that for all $\kk > \omega$, the property of $\kk$-structural completeness is both very strong and rare. Thus, on the one hand, for varieties generated by a finite algebra of finite type, $\kappa$-structural completeness coincides with structural completeness for all infinite $\kappa$ (Theorem~\ref{th-tabular1}). On the other hand, the following holds:

\begin{mytheorem} $\mathbf{\ref{th-fatab}}$\
	Suppose that $\vv V$ is a structurally complete congruence distributive variety of finite type and $\kappa > \omega$. Then $\vv V$ is $\kappa$-structurally complete if and only if it is generated by a finite algebra.
\end{mytheorem}

Thus, even the variety of Heyting algebras $\vv{LC}$ generated by all linearly ordered Heyting algebras, which is structurally complete (see, e.g., \cite{DzikWronski1973}), is not $\kappa$-structurally complete for any $\kappa > \omega$. The same holds for the variety $\vv{ML}$ corresponding to Medvedev logic, which is structurally complete (see \cite{PrucnalTwo1979}).

In Section~\ref{sec-primitive}, we consider hereditarily $\kk$-structurally complete varieties—the $\kk$-primitive varieties. In \cite{Citkin1978}, it was observed that the Heyting algebra $\alg C_{\omega+1}$ of order type $\omega + 1$ defines a non-finitary consequence relation: the prevariety defined by this algebra is not a quasivariety, because the following infinite rule is valid in $\alg C_{\omega+1}$, whereas it can be refuted in the algebra $\alg C_{\omega+2}$ of order type $\omega + 2$:
\begin{align*}
	\mathfrak{r} = \frac{\neg\neg p_0,\ (p_j \to p_i) \to p_j,\ \ 0 \leq i < j \leq \omega}{p_\omega},
\end{align*}
whereas $\QQ(\alg C_{\omega+1}) = \QQ(\alg C_{\omega+2})$.

Note that the rule
\begin{align*}
	\bb = \frac{(p_i \leftrightarrow p_j) \to p_\omega,\ \ 0 \leq i < j < \omega}{p_\omega}
\end{align*}
introduced independently in \cite{PrucnalStructural1985} and \cite{RautenbergNote1985}, and also playing an important role in \cite{WojtylakSyntactical1990, WojtylakStructural1991}, separates aforementioned algebras. Moreover, for every $\kk > \omega$, a primitive variety of Heyting algebras $\vv V$ is $\kk$-primitive if and only if the rule $\mathfrak{r}$ (or $\bb$) is valid in $\vv V$.

\section{Preliminaries}\label{sec-pr}

Throughout the paper, we use an arbitrary but fixed signature $\Omega$ (not containing the symbol $\Ra$); all algebras and terms are assumed to be of this type. Terms are built in the usual way from a set of variables and symbols from $\Omega$, and $Eq$ denotes the set of all equations.

If $\GG \subseteq Eq$ and $\gm \in Eq$, the pair $\GG \Ra \gm$ is called a \Def{sequent}, or an \Def{implication} (e.g., \cite{ShafaatImplicationally1969}), a \Def{generalized quasiequation} (e.g., \cite{MoraschiniLogical2018}), or a \Def{consecution} (e.g., \cite{LavickaNogueraNew2017}). If $\GG$ is empty, we omit the antecedent, and the sequent reduces to an equation. If $\GG$ is finite, the sequent is a \Def{quasiequation}. If $\lm$ and $\kk$ are cardinals such that $|\GG| < \lm$ and the sequent $\GG \Ra \gm$ contains fewer than $\kk$ distinct variables, then it is called a $(\lm,\kk)$-\Def{quasiequation}; if $\lm = \kk$, it is called a $\kk$-\Def{quasiequation}.

If $\alg A$ is an algebra, a map from the set of all variables into $\alg A$ is called a \Def{valuation} in $\alg A$. Clearly, any valuation $\nu$ in $\alg A$ can be extended to a map from the set of all terms into $\alg A$: for a term $t(x_1,\dots,x_n)$, we define
\[
\nu\bigl(t(x_1,\dots,x_n)\bigr) = t\bigl(\nu(x_1),\dots,\nu(x_n)\bigr).
\]
We say that $\nu$ \Def{validates} the equation $t \app s$ (in symbols, $\alg A \models_\nu t \app s$) if $\nu(t) = \nu(s)$. If $\GG$ is a set of equations, we write $\alg A \models_\nu \GG$ to indicate that $\nu$ validates every equation in $\GG$. Throughout the paper, we assume that $\kk \ge \omega$ is a regular cardinal.

Given an algebra $\alg A$ and a $\kk$-quasiequation $\GG \Ra \gm$, we say that this $\kk$-quasiequation is \Def{valid} in $\alg A$ (in symbols, $\alg A \models \GG \Ra \gm$) if, for every valuation $\nu$ in $\alg A$,
\begin{align*}
	\alg A \models_\nu \GG \ \text{implies} \ \alg A \models_\nu \gm.
\end{align*}
A $\kk$-quasiequation $\sq$ is \Def{valid in a class of algebras} $\vv K$ if it is valid in every algebra from $\vv K$. A set of $\kk$-quasiequations $\Sq$ is valid in $\vv K$ if every $\kk$-quasiequation in $\Sq$ is valid in $\vv K$. Two sets of $\kk$-quasiequations that are valid in exactly the same algebras from $\vv K$ are called $\vv K$-\Def{equivalent}.

Observe that, since a $\kk$-quasiequation contains fewer than $\kk$ distinct variables, a $\kk$-quasiequation $\sq$ is valid in an algebra $\alg A$ if and only if it is valid in every $\lm$-generated subalgebra of $\alg A$ for all $\lm < \kk$.
\smallskip

Let $\vv K$ be a class of algebras. As usual, $\II$, $\HH$, $\SU$, $\PP$, and $\PP_u$ denote the class operators: $\II(\vv K)$, $\HH(\vv K)$, $\SU(\vv K)$, $\PP(\vv K)$, and $\PP_u(\vv K)$ are, respectively, the classes of all isomorphic copies, homomorphic images, subalgebras, Cartesian products (including the product over the empty set), and ultraproducts of algebras from $\vv K$. Classes of algebras closed under $\HH$, $\SU$, and $\PP$ are called \Def{varieties}; classes closed under $\II$, $\SU$, $\PP$, and $\PP_u$ are called \Def{quasivarieties}; and classes closed under $\II$, $\SU$, and $\PP$ are called \Def{prevarieties}.

It is well known that every variety can be defined by a set of equations, and every set of equations defines a variety. Similarly, every quasivariety can be defined by a set of quasiequations, and every set of quasiequations defines a quasivariety. A class of algebras defined by a set of $\kk$-quasiequations is called a $\kk$-\Def{quasivariety}.

Given a class of algebras $\vv K$, we use $\VV(\vv K)$, $\QQ(\vv K)$, and $\PR(\vv K)$ to denote, respectively, the smallest variety, quasivariety, and prevariety containing $\vv K$; we say that these classes are \Def{generated} by $\vv K$. It is well known that
\[
\VV(\vv K)=\HH\SU\PP(\vv K), \qquad
\QQ(\vv K)=\II\SU\PP\PP_u(\vv K), \qquad
\PR(\vv K)=\II\SU\PP(\vv K).
\]

A class of algebras $\vv K$ is called $\kk$-\Def{local} if, for every algebra $\alg A$, whenever every $\lm$-generated subalgebra of $\alg A$ (for all $\lm < \kk$) belongs to $\vv K$, it follows that $\alg A \in \vv K$. Clearly, $\kk$-quasivarieties are $\kk$-local classes (cf.\ \cite[Section~2.3.4]{GorbunovBookE}).

\medskip

Let us introduce an additional class operator (similar to the one in \cite{BlokJonsson1999}): for any cardinal $\kappa \ge \omega$ and any class of algebras $\vv K$, let $\UUk(\vv K)$ denote the class of all algebras whose every $\lm$-generated subalgebra, for $\lm < \kk$, belongs to $\vv K$.

Given a class of algebras $\vv K$ and an infinite cardinal $\kk$, denote by $\QQ_\kk(\vv K)$ the smallest $\kk$-quasivariety containing $\vv K$. Then
\[
\QQ_\kk(\vv K) = \UUk,\II\SU\PP(\vv K).
\]
Note that $\QQ_\kk(\vv K)$ is always a prevariety, although $\QQ_\kk(\vv K)$ and $\PR(\vv K)$ may be distinct. It is also clear that for any $\omega \leq \kk_1 \leq \kk_2$,
\[
\QQ_{\kk_2}(\vv K) \subseteq \QQ_{\kk_1}(\vv K) \subseteq \QQ_\omega(\vv K) = \QQ(\vv K).
\]

The following notion will play an important role in what follows. An algebra $\alg A \in \vv K$ is \Def{weakly projective} in $\vv K$ (or weakly $\vv K$-projective) if $\alg A$ embeds into each of its homomorphic preimages in $\vv K$; that is, whenever $\alg A \in \HH(\alg B)$ for some $\alg B \in \vv K$, then $\alg A \in \II\SU(\alg B)$. For example, every free algebra in a variety $\vv V$ is weakly $\vv V$-projective (indeed, projective; see, e.g., \cite{GraetzerB}).

\paragraph{Congruences and irreducubility.}

\begin{definition}
	Let $\alg A$ be a nontrivial algebra. Denote by $\mu_{\alg A}$ the meet of all congruences on $\alg A$ distinct from the identity congruence, and call it the \Def{monolith of} $\alg A$. Algebra $\alg A$ is said to be \Def{subdirectly irreducible (SI)} if $\mu_{\alg A}$ is not the identity congruence.
\end{definition}

\begin{definition}
	Let $\vv K$ be a class of algebras and $\alg A \in \vv K$.
	Congruence $\vt$ on $\alg A$ is a \Def{$\vv K$}-congruence provided $\alg A/\vt \in \vv K$.
	Denote by $\mu_\vv K(\alg A)$ the meet of all $\vv K$-congruences on $\alg A$ distinct from the identity congruence, and call it the $\vv K-$\Def{monolith of} $\alg A$. Algebra $\alg A$ is said to be \Def{$\vv K$-irreducible} if $\mu_\vv K(\alg A)$ is not the identity congruence. We will omit the reference to the class if no confusion arises. 
\end{definition}
Clearly, if  $\vv K$ is a variety, then $\vv K$-irreducible algebras are precisely the subdirectly irreducible (SI) algebras.

As usual, if $a,b \in \alg A$, we denote by $\vt_{\alg A}(a,b)$ the principal congruence on $\alg A$ generated by the pair $(a,b)$, and we omit the reference to $\alg A$ when it is clear from the context.

\begin{prop}\label{pr-qirr} 
	Let $\vv Q$ be a quasivariety. Then the following hold:
	\begin{enumerate}
		\item[(a)] Every nontrivial algebra in $\vv Q$ is a subdirect product of $\vv Q$-irreducible algebras;
		
		\item[(b)] If $\vv Q$ has finite type and $\vv Q = \QQ(\vv K)$, where $\vv K$ is a class of algebras, then every finite $\vv Q$-irreducible algebra belongs to $\II\SU(\vv K)$;
		
		\item[(c)] If $\vv Q = \QQ(\alg A)$ and $\alg A$ is finite, then all $\vv Q$-irreducible algebras belong to $\II\SU(\alg A)$ and are therefore finite.
	\end{enumerate} 
\end{prop}

For (a), see, e.g., \cite[Theorem 3.1.1]{GorbunovBookE}; (b) follows from \cite[Proposition 2.1.15]{GorbunovBookE}; and (c) follows from the observation that, by \cite[Lemma 1.5]{CzelakowskiDziobiakCongruence1990}, for any irreducible algebra $\alg B \in \QQ(\alg A)$, we have $\alg B \in \II\SU\PPu(\alg A)$, and $\II\SU\PPu(\alg A) = \II\SU(\alg A)$ because $\alg A$ is finite.

\begin{cor}\label{cor-countSI}
	If $\vv P$ is a proper subprevariety of a quasivariety $\vv Q$, then $\vv Q \setminus \vv P$ contains a countable $\vv Q$-irreducible algebra.
\end{cor}

\begin{proof}
	Let $\vv{P} \subsetneq \vv Q$ be a proper subprevariety. Recall that every algebra in $\vv{Q}$ is a subdirect product of $\vv{Q}$-irreducible algebras. Since the prevariety $\vv{P}$ is closed under subdirect products and $\vv{P} \subsetneq \vv{Q}$, the difference $\vv{Q}\setminus \vv{P}$ must contain a $\vv{Q}$-irreducible algebra.
\end{proof}

A variety for which every SI member has cardinality less than $m$, for some $m < \omega$, is said to have a \Def{finite residual bound} (see, e.g., \cite[Section~4.3]{McKenzieetBookv2}). Recall the following property of congruence meet-semidistributive varieties.

\begin{theorem}[{\cite[Theorem~4.1]{KearnesWillardResidually1999}}]
	Suppose that $\vV$ is a congruence meet-semidistributive variety of finite type. If $\vV$ contains arbitrarily large finite SI algebras, then $\vV$ contains an infinite SI algebra.
\end{theorem}

We will use this in the following form.
\begin{cor}\label{cor-KW}
	If a congruence meet-semidistributive variety of finite type is not residually $m$-small for some finite $m$, then it contains an infinite SI algebra.
\end{cor}

\paragraph{Characteristic $\kk$-quasiequations.}
With every $\vv K$-irreducible algebra $\alg A \in \vv K$ of finite type and cardinality at most $\kk$, we associate a $\kk$-quasiequation $\csq_\alg A$, which we call a \Def{characteristic $\kk$-quasiequation} of $\alg A$ (an infinite analogue of the Jankov characteristic formula). To each element $a \in \alg A$, we assign a variable $x_a$. Since $\alg A$ is $\vv K$-irreducible, the monolith $\mu_{\vv K}(\alg A)$ is nontrivial and contains a pair $(c,d)$ with $c \neq d$. For each operation $f \in \Omega$, let $k_f$ denote the arity of $f$. We then define
\begin{align*}
	\csq_\alg A =
	\set{f(x_{a_1},\dots,x_{a_{k_f}}) \app x_{f(a_1,\dots,a_{k_f})}}
	{\text{for all } f \in \Omega \text{ and all } a_1,\dots,a_{k_f} \in \alg A}
	\Ra x_c \app x_d.
\end{align*}
Observe that the valuation $\nu \colon x_a \mapsto a$ refutes $\csq_\alg A$ in $\alg A$; that is, $\alg A \not\models \csq_\alg A$.

The following proposition justifies the name “characteristic.”

\begin{prop}\label{pr-char}
	Let $\vv K$ be a class of algebras of finite type, and let $\alg A \in \vv K$ be a $\vv K$-irreducible algebra of cardinality at most $\kk$. Then, for every algebra $\alg B \in \vv K$,
	\begin{align*}
		\alg B \not\models \csq_\alg A
		\quad \text{if and only if} \quad
		\alg A \in \II\SU(\alg B).
	\end{align*}
\end{prop}

\begin{proof}
	The implication from right to left is immediate.
	
	For the converse, suppose that there exists a valuation $\nu$ that refutes $\csq_\alg A$ in $\alg B$. Then all premises of $\csq_\alg A$ are satisfied under $\nu$, while $\nu(x_c) \neq \nu(x_d)$. Define a map $h \colon \alg A \to \alg B$ by $h(a) = \nu(x_a)$. Since the premises are satisfied, $h$ preserves all operations in $\Omega$, and hence $h$ is a homomorphism from $\alg A$ into $\alg B$.
	
	Moreover, since $\nu$ refutes $\csq_\alg A$, we have $h(c) = \nu(x_c) \neq \nu(x_d) = h(d)$. Thus $h(c) \neq h(d)$, and since $(c,d) \in \mu_{\vv K}(\alg A)$, it follows from the definition of the $\vv K$-monolith that $h$ is injective. Therefore, $h$ is an embedding, and hence $\alg A \in \II\SU(\alg B)$.
\end{proof}

\begin{example}\label{ex-chains}
	Consider the Heyting algebras $\alg C_{\omega+1}$ and $\alg C_{\omega+2}$ whose Hasse diagrams are depicted in Fig.~\ref{fig-chain}.
	
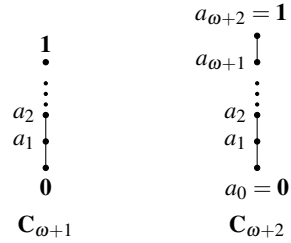
\begin{figure}[ht]
	\centering
	\begin{tikzpicture}[scale=.7]
		\draw  (0,0) -- (0,1);
		\draw[fill] (0,0) circle [radius=0.05];
		\draw[fill] (0,0.5) circle [radius=0.05];
		\draw[fill] (0,1) circle [radius=0.05];
		\draw[fill] (0,1.2) circle [radius=0.03];
		\draw[fill] (0,1.4) circle [radius=0.03];
		\draw[fill] (0,1.6) circle [radius=0.03];
		\draw[fill] (0,2) circle [radius=0.05];
		\node[below] at (0,0) {\footnotesize $\zero$};
		\node[left] at (0,0.5) {\footnotesize $a_1$};
		\node[left] at (0,1) {\footnotesize $a_2$};
		\node[above] at (0,2) {\footnotesize $\one$};	
		\node[below] at (0,-0.7) {\footnotesize $\alg{C}_{\omega+1}$};	
		\draw  (4,0) -- (4,1);
		\draw  (4,2) -- (4,2.5);
		\draw[fill] (4,0) circle [radius=0.05];
		\draw[fill] (4,0.5) circle [radius=0.05];
		\draw[fill] (4,1) circle [radius=0.05];
		\draw[fill] (4,1.2) circle [radius=0.03];
		\draw[fill] (4,1.4) circle [radius=0.03];
		\draw[fill] (4,1.6) circle [radius=0.03];
		\draw[fill] (4,2) circle [radius=0.05];
		\draw[fill] (4,2.5) circle [radius=0.05];
		\node[below] at (4,0) {\footnotesize $a_0 = \zero$};
		\node[left] at (4,0.5) {\footnotesize $a_1$};
		\node[left] at (4,1) {\footnotesize $a_2$};
		\node[left] at (4,2) {\footnotesize $a_{\omega+1}$};
		\node[above] at (3.7,2.5) {\footnotesize $a_{\omega+2} = \one$};	
		\node[below] at (4,-0.7) {\footnotesize $\alg{C}_{\omega+2}$};	
	\end{tikzpicture}
	\caption{Chain Heyting algebras.}
	\label{fig-chain}
\end{figure}

It is not hard to see that the algebra $\alg C_{\omega+2}$ is SI, since $(a_{\omega+1}, a_{\omega+2}) \in \mu(\alg C_{\omega+2})$. Therefore, it has a characteristic $\kk$-quasiequation for every $\kk > \omega$: we take variables $x_{a_i}$, where $i \in [0, \omega+2]$, and define
\begin{align*}
	\csq := \quad {\neg\neg x_{a_0} \approx x_{a_{\omega+2}},\ \ x_{a_j} \to x_{a_i} \approx x_{a_i},\ \ 0 \leq i < j \leq \omega+2}
	\ \Rightarrow\
	x_{a_{\omega+1}} \approx x_{a_{\omega+2}}.
\end{align*}

Clearly, the valuation $x_{a_i} \mapsto a_i$ refutes $\csq$ in $\alg C_{\omega+2}$; that is, $\alg C_{\omega+2} \not\models \csq$.

On the other hand, $\alg C_{\omega+2}$ cannot be embedded into $\alg C_{\omega+1}$, and by Proposition~\ref{pr-char}, it follows that
\[ 
\alg{C}_{\omega+1} \models \csq. 
\]
Thus, for every $\kk > \omega$,
\begin{align}
	\alg C_{\omega+2} \notin \QQ_\kk(\alg C_{\omega+1}).
\end{align}
\end{example}

\begin{remark}
In Example~\ref{ex-chains}, we simplified the $\kk$-characteristic quasiequation by omitting from the antecedent some equations that are consequences of the remaining ones. For example, if $x_{a_j} \to x_{a_i} = x_{a_i}$, then $x_{a_i} \leq x_{a_j}$ and therefore $x_{a_i} \land x_{a_j} = x_{a_i}$, $x_{a_i} \lor x_{a_j} = x_{a_j}$, and $x_{a_i} \to x_{a_j} = x_{\one}$. Moreover, it is not hard to see that the equations $\neg x_{a_i} = x_{\zero}$ can be omitted for all $a_i \neq a_{\zero}$.
\end{remark}

\paragraph{Congruence extension property.}

An algebra $\alg A$ has the \Def{congruence extension property} (CEP) if, for every subalgebra $\alg B \in \SU(\alg A)$ and every congruence $\theta$ on $\alg B$, there exists a congruence $\theta'$ on $\alg A$ such that $\theta = \theta' \cap B^2$. A class of algebras has the CEP if every algebra in the class has the CEP.

In general, a subalgebra of an SI algebra need not be SI; indeed, consider the algebras $\alg C_{\omega+1}$ and $\alg C_{\omega+2}$ from Example~\ref{ex-chains}: $\alg C_{\omega+1} \in \II\SU(\alg C_{\omega+2})$, but $\alg C_{\omega+2}$ is SI, whereas $\alg C_{\omega+1}$ is not. Nevertheless, the following holds.

\begin{prop}\label{pr-CEPirr}
	Let $\vv K$ be a class of algebras, let $\alg A \in \vv K$ be a $\vv K$-irreducible algebra with the CEP, and let $\alg B$ be a subalgebra of $\alg A$ containing two distinct elements from $\mu_{\vv K}(\alg A)$. Then $\alg B$ is $\vv K$-irreducible.
\end{prop}

\begin{proof}
	Suppose $(a,b) \in \mu_{\vv K}(\alg A)$ with $a \neq b$ and $a,b \in \alg B$. Assume, for a contradiction, that $\alg B$ is not $\vv K$-irreducible. Then $\mu_{\vv K}(\alg B)$ is the identity congruence, and hence there exists a $\vv K$-congruence $\theta$ on $\alg B$ such that $(a,b) \notin \theta$. Since $\alg A$ has the CEP, there exists a congruence $\theta'$ on $\alg A$ such that $\theta = \theta' \cap B^2$. Consequently, $(a,b) \notin \theta' \cap B^2$, which implies $(a,b) \notin \theta'$, contradicting the assumption that $\alg A$ is $\vv K$-irreducible.
\end{proof}

\begin{cor}\label{cor-CEPsubIrr}
	If $\vv K$ is a class of algebras of finite type, then every $\vv K$-irreducible algebra with the CEP contains a countable $\vv K$-irreducible subalgebra. Moreover, if $\alg A$ is an infinite $\vv K$-irreducible algebra of cardinality $\lm$, then for every infinite cardinal $\lm' \leq \lm$, there exists a $\vv K$-irreducible subalgebra $\alg B \in \II\SU(\alg A)$ of cardinality $\lm'$.
\end{cor}

\begin{proof}
	Since $\alg A$ is $\vv K$-irreducible, there exist distinct elements $a,b \in \alg A$ such that $(a,b) \in \mu_{\vv K}(\alg A)$. The subalgebra generated by ${a,b}$ is countable because the similarity type is finite, and it is $\vv K$-irreducible by Proposition~\ref{pr-CEPirr}.

	If $\alg A$ is infinite of cardinality $\lm$ and $\lm' \leq \lm$, take $\lm'$ distinct elements of $\alg A$, adjoin two distinct elements from $\mu_{\vv K}(\alg A)$, and consider the subalgebra $\alg B$ generated by this set. Since the type is finite, $|\alg B| = \lm'$, and by Proposition~\ref{pr-CEPirr}, $\alg B$ is $\vv K$-irreducible.
\end{proof}

\begin{cor}\label{cor-CEPref}
	Let $\vv K$ be a class of algebras of finite type. If a $\kk$-quasiequation $\sq$ is refuted in a $\vv K$-irreducible algebra $\alg A \in \vv K$, then there exists a $\vv K$-irreducible subalgebra $\alg B \in \SU(\alg A)$ of cardinality $\kk$ that also refutes $\sq$.
\end{cor}

\begin{proof}
	Suppose that $\alg A$ is $\vv K$-irreducible and $\alg A \not\models \sq$, and let $a,b \in \alg A$ with $(a,b) \in \mu_{\vv K}(\alg A)$.
	
	Let $x_k$, $k \in \kk$, be all variables occurring in $\sq$, and let $\nu$ be a valuation in $\alg A$ that refutes $\sq$. Then the subalgebra $\alg B$ of $\alg A$ generated by the elements $a$, $b$, and $\nu(x_k)$ (for $k \in \kk$) has cardinality at most $\kk$, since $\alg A$ has finite type. By Proposition~\ref{pr-CEPirr}, $\alg B$ is $\vv K$-irreducible. Clearly, $\nu$ is a valuation refuting $\sq$ in $\alg B$.
\end{proof}

\paragraph{Free algebras.}

Suppose that $\vv K$ is a class of algebras and $\alg F \in \vv K$ is generated by a set $G \subseteq \alg F$. Then $G$ \Def{freely generates} $\alg F$ \Def{relative to} $\vv K$ if, for every algebra $\alg A \in \vv K$, any map $f \colon G \to \alg A$ extends to a homomorphism from $\alg F$ into $\alg A$. In this case, $\alg F$ is called a $\vv K$-\Def{free algebra of rank} $\rho$, where $\rho = |G|$. The $\vv K$-free algebra of rank $\rho$ is denoted by $\alg F_\vv K(\rho)$, and the class of all $\vv K$-free algebras is denoted by $\vv K_F$. Free $\vv K$-algebras of a given rank are unique up to isomorphism (cf.\ \cite[Theorem~10.7]{BurrisSanka}).

One of the most important properties of prevarieties is that every nontrivial prevariety contains free algebras of all ranks (see, e.g., \cite[Theorem~10.12]{BurrisSanka}). This yields the following observation, which will be important in what follows.

\begin{prop}\label{pr-free}
	For every class of algebras $\vv K$,
	\begin{align*}
		\PR(\vv K)_F = \QQ(\vv K)_F = \VV(\vv K)_F.
	\end{align*}
\end{prop}

\begin{proof}
	If $\vv K$ contains only trivial algebras, then the statement is immediate, since $\PR(\vv K) = \QQ(\vv K) = \VV(\vv K)$. Now assume that $\vv K$ contains a nontrivial algebra. Then the classes $\VV(\vv K)$, $\QQ(\vv K)$, and $\PR(\vv K)$—being prevarieties—all contain free algebras of arbitrary ranks.

	Moreover, by \cite[Corollary~2.1.13]{GorbunovBookE}, for any class of algebras $\vv K$, every $\vv K$-free algebra is free in $\VV(\vv K)$. Since $\VV(\vv K) \supseteq \QQ(\vv K) \supseteq \PR(\vv K)$, it follows that every $\VV(\vv K)$-free algebra is also free in $\QQ(\vv K)$ and $\PR(\vv K)$. By uniqueness (up to isomorphism) of free algebras of a given rank, the free algebras in $\PR(\vv K)$, $\QQ(\vv K)$ and $\VV(\vv K)$ coincide. 
\end{proof}


\begin{cor}\label{cor-free}
	Let $\vv V$ be a variety. Then, for every infinite cardinal $\kk$, the class $\QQ_\kk(\FrV)$ is the smallest $\kk$-quasivariety $\vv Q_\kk$ such that $\VV(\vv Q_\kk)=\vv V$, and $\PR(\FrV)$ is the smallest prevariety $\vv P$ such that $\VV(\vv P) = \vv V$.
\end{cor}

\begin{proof}
	First, observe that since $\VV(\FrV) = \vv V$, we have $\VV(\QQ_\kk(\FrV)) = \vv V$.

	Next, suppose that $\vv K$ is a $\kk$-quasivariety such that $\VV(\vv K) = \vv V$. Then $\vv K$ contains all $\vv V$-free algebras, and hence $\FrV \in \vv K$. Therefore, $\QQ_\kk(\FrV) \subseteq \vv K$.
	
	The statement for prevarieties follows by a similar argument.
	\end{proof}
	
Clearly, in contrast to varieties, prevarieties and quasivarieties need not be generated by their free algebras. Nevertheless, the following holds.

\begin{cor}
	For any prevariety $\vv K$, $\VV(\vv K_F) = \VV(\vv K)$.
\end{cor}
\begin{proof}
	Let $\vv V = \VV(\vv K)$. Then $\vv K_F = \vv V_F$, therefore
	\begin{align*}
		\VV(\vv K_F) = \VV(\vv V_F) = \vv V = \VV(\vv K).
	\end{align*}
\end{proof}


Moreover, the following holds.

\begin{theorem}\label{th-freesame}
	For any prevariety $\vv K$ and any infinite cardinal $\kk$,
	\begin{align*}
		\VV(\alg F_\vv K(\omega)) = \VV(\vv K_F), \quad
		\QQ_\kk(\alg F_\vv K(\omega)) = \QQ_\kk(\vv K_F), \quad
		\PR(\alg F_\vv K(\omega)) = \PR(\vv K_F).
	\end{align*}
	In particular, for any variety $\vv V$,
	\begin{align*}
		\VV(\FrV) = \VV(\vv V_F), \quad
		\QQ_\kk(\FrV) = \QQ_\kk(\vv V_F), \quad
		\PR(\FrV) = \PR(\vv V_F).
	\end{align*}
\end{theorem}

\begin{proof}
	Let $\vv W = \VV(\vv K)$. Since $\vv K$ is a prevariety, $\vv W_F = \vv K_F$; in particular, $\alg F_\vv K(\omega) \cong \alg F_\vv W(\omega)$. Thus,
	\begin{align*}
		\VV(\alg F_\vv K(\omega)) = \VV(\alg F_\vv W(\omega)) = \VV(\vv W_F) = \VV(\vv K_F).
	\end{align*}
	
	Next, since $\alg F_\vv K(\omega) \in \vv K_F$, we have
	\begin{align*}
		\QQ_\kk(\alg F_\vv K(\omega)) \subseteq \QQ_\kk(\vv K_F) \quad \text{and} \quad
		\PR(\alg F_\vv K(\omega)) \subseteq \PR(\vv K_F),
	\end{align*}
	and it remains to prove the reverse inclusions.
	
	Since $\vv K_F = \vv W_F \subseteq \vv W$, we have $\vv K_F \subseteq \VV(\alg F_\vv K(\omega))$. By Proposition~\ref{pr-free},
	\begin{align*}
		\VV(\alg F_\vv K(\omega))F = \PR(\alg F\vv K(\omega))F,
	\end{align*}
	and therefore
	\begin{align*}
		\vv K_F \subseteq \VV(\alg F\vv K(\omega))F = \PR(\alg F\vv K(\omega))F \subseteq \QQ\kk(\alg F_\vv K(\omega)).
	\end{align*}
	Thus,
	\begin{align*}
		\PR(\vv K_F) \subseteq \PR(\alg F_\vv K(\omega)) \quad \text{and} \quad
		\QQ_\kk(\vv K_F) \subseteq \QQ_\kk(\alg F_\vv K(\omega)).
	\end{align*}
\end{proof}

\subsubsection*{Varieties with EDPC.}

Let us recall from \cite{BlkPgz1} the definition of varieties with equationally definable principal congruences and some of their properties that will be used later.

\begin{definition}
	A variety $\vv V$ is said to have \Def{equationally definable principal congruences (EDPC)} if there exist $4$-ary terms
	$p_i(x,y,z,w), q_i(x,y,z,w), i = 1,\dots,n$,
	such that for every algebra $\alg A \in \vv V$ and all $a,b,c,d \in \alg A$,
	\begin{align}
		(c,d) \in \vartheta(a,b) \text{ if and only if } \alg A \models p_i(a,b,c,d) \app q_i(a,b,c,d), \quad 1 \leq i \leq n.
		\label{def-EDPC}
	\end{align}
\end{definition}

To simplify notation, we let
$E(x,y,z,w) := \set{ p_i(x,y,z,w) \app q_i(x,y,z,w)}{ i = 1,\dots,n }$,
and we can rephrase \eqref{def-EDPC} in the following way:
\begin{align}
	(c,d) \in \vartheta(a,b) \text{ if and only if } \alg A \models E(a,b,c,d).
	\label
{ex-EDPC}
\end{align}

Recall that every variety with EDPC is congruence-distributive and has the CEP (see \cite[Theorem~1.2]{BlkPgz1}).

\subsubsection*{Bounding $\kk$-quasiequation.}
Let $\vv V$ be a variety with EDPC witnessed by $E(x,y,z,w)$, and let $\kk$ be an infinite cardinal.

Consider the following $\kk^+$-quasiequation, which is a generalization of the rule introduced in \cite{PrucnalStructural1983} and used in \cite{PrucnalStructural1985, RautenbergNote1985, WojtylakSyntactical1990, DzikWojtylakAlmost2015}:
\begin{align}
	\pVk := \seq{\set{E(x_j,x_k,z,w)}{0 \leq j < k < \kk}}{z \app w}.
\end{align}

\begin{theorem}\label{th-infchar}
	Let $\vv V$ be a variety with EDPC and $\kk$ be an infinite cardinal. Then:
	\begin{enumerate}
		\item[(i)] $\pVk$ is valid in every algebra from $\vv V$ of cardinality less than $\kk$;
		\item[(ii)] $\pVk$ fails in every SI algebra from $\vv V$ of cardinality at least $\kk$.
	\end{enumerate}
\end{theorem}

\begin{proof}
	(i) Suppose that $\alg A \in \vv V$ and $|\alg A| < \kk$. Then the number of distinct pairs $(a,b) \in \alg A^2$ is less than $\kk$. Therefore, for any valuation $\nu$ in $\alg A$, there exist $j,k$ with $0 \le j < k < \kk$ such that $\nu(x_j) = \nu(x_k)$, and by \eqref
{ex-EDPC} we obtain $\nu(z) = \nu(w)$. Hence, $\alg A \models \pVk$.
	\smallskip
	
	(ii) Suppose that $\alg A \in \vv V$ is an SI algebra with $|\alg A| \ge \kk$. Then the monolith $\mu(\alg A)$ contains a pair $(c,d)$ with $c \neq d$. Since $|\alg A| \ge \kk$, there exist pairwise distinct elements $a_j \in \alg A$ for $0 \le j < \kk$. Consider the valuation
	\begin{align*}
		\nu(x_j) = a_j \ (0 \le j < \kk), \qquad \nu(z) = c, \qquad \nu(w) = d.
	\end{align*}
	Since $(c,d) \in \mu(\alg A)$, we have $(c,d) \in \theta(a,b)$ for all $a,b \in \alg A$ with $a \neq b$. Consequently, for every equation $p_i(x,y,z,w) \app q_i(x,y,z,w) \in E(x,y,z,w)$, we have
	\begin{align*}
		p_i(a_j,a_k,c,d) = q_i(a_j,a_k,c,d)
	\end{align*}
	for all $0 \le j < k < \kk$. Thus, all premises of $\pVk$ are satisfied under $\nu$, while $\nu(z) \neq \nu(w)$, and hence $\nu$ refutes $\pVk$.
	
\end{proof}

\begin{example}
	The variety of all Heyting algebras $\HA$ has EDPC witnessed by
	$(x \leftrightarrow y) \to z\app (x \leftrightarrow y) \to w$.
	Thus, the bounding $\kk$-quasiequation for $\HA$ is
	\begin{align*}
		\mathfrak{b}_\HA(\omega) =
		\seq{\set{(x_j \leftrightarrow x_k) \to z \app (x_j \leftrightarrow x_k) \to w}{0 \le j < k < \omega}}{z \app w}.
	\end{align*}
	Clearly, $\mathfrak{b}\HA(\omega)$ is valid in every finite Heyting algebra, and it is refuted in every infinite SI Heyting algebra $\alg A$. Indeed, let $a_i$, for $i=1,2,\dots$, be distinct elements of $\alg A$, and let $b$ be a pre-top element, which exists because $\alg A$ is SI. Then $(a_i \leftrightarrow a_j) \to b = \one$ for all distinct $a_i, a_j$, whereas $b \neq \one$. Thus, the valuation $\nu$ that sends $x_i$ to $a_i$, $z$ to $b$, and $w$ to $\one$ refutes $\mathfrak{b}_\HA(\omega)$.
\end{example}

\section{Strong structural completeness}\label{sec-SSCpl}

\paragraph{Admissible and derivable $\kk$-quasiequations.}

If $\vv Q$ is a $\kk$-quasivariety, a $\kk$-quasiequation $\sq$ is said to be \Def{derivable} in $\vv Q$ provided that $\vv Q \models \sq$, and $\sq$ is \Def{admissible} in $\vv Q$ if $\VV(\vv Q) = \VV(\vv Q^\sq)$, where $\vv Q^\sq$ is the class of all algebras from $\vv Q$ in which $\sq$ is valid. In other words, $\sq$ is admissible in $\vv Q$ if the set of all equations valid in $\vv Q$ is closed under $\sq$. Thus, admissibility of $\sq$ in $\vv Q$ amounts to validity of $\sq$ in $\alg F_{\VV(\vv Q)}(\omega)$; hence, in view of Proposition~\ref{pr-free}, $\sq$ is admissible in $\vv Q$ if and only if $\alg F_\vv Q(\omega) \models \sq$.

In other words, if $\sq$ is admissible in $\vv Q$, it may happen that $\vv Q^\sq \subsetneq \vv Q$, but the equations valid in both $\kk$-quasivarieties coincide.


\begin{prop}\label{pr-adm}
	Let $\vv K$ be a class of algebras and let $\sq$ be a $\kk$-quasiequation. Then $\sq$ is admissible in $\VV(\vv K)$ (respectively, for $\QQ_\kk(\vv K)$) provided that $\vv K \models \sq$.
\end{prop}

\begin{proof}
	Suppose that $\vv K \models \sq$. Let $\vv V = \VV(\vv K)$ and assume, for a contradiction, that $\FrV \not\models \sq$.
	
	Let $\vv P \subseteq \vv V$ be the subclass defined (relative to $\vv V$) by $\sq$:
	\begin{align*}
		\vv P = \{ \alg A \in \vv V \mid \alg A \models \sq \}.
	\end{align*}
	Since $\vv K \models \sq$, we have $\vv K \subseteq \vv P$, and therefore $\vv V = \VV(\vv K) \subseteq \VV(\vv P) \subseteq \vv V$, whence $\vv V = \VV(\vv P)$. By Corollary~\ref{cor-free}, $\PR(\FrV) \subseteq \vv P$, and hence $\FrV \in \vv P$, contradicting the assumption that $\FrV \not\models \sq$.
	
	Thus, $\FrV \models \sq$. By Proposition~\ref{pr-free}, $\FrV \cong \alg F_{\vv Q}(\omega)$ (where $\vv Q = \QQ_\kk(\vv K)$), and hence $\alg F_{\vv Q}(\omega) \models \sq$. Therefore, $\sq$ is admissible in $\vv Q$.
\end{proof}

\paragraph{Structural completeness}

\begin{definition}\label{def-SCpl}
Let $\vv V$ be a variety and $\kappa$ be an infinite cardinal. Then $\vV$ is

\begin{enumerate}
	\item[] \Def{structurally complete (SCpl)} if every admissible in $\vV$ quasiequation is derivable in $\vV$, that is, $ \vv V = \QQ(\FrV)$;
	
	\item[] \Def{strongly structurally complete (SSCpl)} if for every infinite cardinal $\kk$, every admissible in $\vV$ $\kk$-quasiequation is derivable in $\vV$, that is, $\vv V = \PR(\FrV)$;
	
	\item[] \Def{$\kappa$-structurally complete ($\kk$-SCpl)} if every admissible in $\vV$ $\kk$-quasiequation is derivable in $\vV$, that is, $\vv V = \QQ_\kappa(\FrV)$.
	
\end{enumerate}

Quasivariety $\vv Q$ is 

\begin{enumerate}
\item[] \Def{strongly structurally complete (SSCpl)} if for every infinite cardinal $\kk$, every admissible in $\vv{Q}$ $\kk$-quasiequation is derivable in $\vv{Q}$, that is, $\vv Q = \PR(\alg F_\vv Q(\omega))$.

\item[] \Def{$\kk$-structurally complete ($\kk$-SCpl)} if every admissible in $\vv{Q}$ $\kk$-quasiequation is derivable in $\vv{Q}$, that is, $\vv Q = \QQ_\kk(\alg F_\vv Q(\omega))$.

\end{enumerate}
\end{definition}
Let us observe that since $\QQ = \QQ_\omega$, structural completeness and $\omega$-structural completeness coincide.
\smallskip

If $\vv V$ is a variety, the $\kk$-quasivariety $\QQ_\kk(\FrV)$ is called the $\kk$-\Def{structural core} of $\vv V$ and is denoted by $\Vk$, and the prevariety $\PR(\FrV)$ is called the \Def{strong structural core} of $\vv V$ and is denoted by $\VP$. Observe that if $\vV$ is strongly structurally complete, then $vV$ is $\kk$-structurally complete for all infinite $\kk$. Also, if $\kk_1 \leq \kk_2$ are infinite cardinals, then $\vV^{\kk_2} \subseteq \vV^{\kk_1}$, and hence , $\kk_1$-incompleteness entails $\kk_2$-incompleteness. Thus, if a variety $\vV$ is $\omega^+$-incomplete, then it is $\kk$-incomplete for all $\kk \ge \omega^+$, and it is not strongly structurally complete.


Clearly, $\omega$-structural completeness coincides with ordinary structural completeness, and, as we shall show, even $\omega^+$-structural completeness is a very strong property. In particular, the remainder of this section is devoted to proving the following theorem.

\begin{theorem} \label{th-fatab}
	Suppose that $\vV$ is a structurally complete congruence distributive variety of finite type generated by finite algebras, and $\kappa > \omega$. Then the variety $\vV$ is $\kappa$-structurally complete if and only if it is tabular.
\end{theorem}

This theorem follows from the following two theorems.

\begin{theorem}\label{th-tabular1}
	Let $\vv K$ be a finite set of finite algebras of finite type. Then $\PR(\vK) = \QQ(\vK)$.
\end{theorem}

Since every finitely generated variety $\vV$ is generated by a single finite $\vV$-free algebra $\alg{F}$, the following corollary holds, and it yields the implication from left to right in Theorem~\ref{th-fatab}.

\begin{cor}\label{cor-tab}
	Let $\vV$ be a finitely generated variety of finite type. Then $\vV$ is SCpl if and only if it is SSCpl.
\end{cor}

The implication from left to right in Theorem~\ref{th-fatab} also follows from the following theorem together with the observation that, by Jónsson’s Lemma, congruence distributive varieties of finite type have a finite residual bound if and only if they are tabular.

\begin{theorem}\label{th-fmp}
	Let $\vv V$ be an SCpl congruence meet-semidistributive variety of finite type generated by finite algebras. Then $\vv V$ is SSCpl if and only if it has a finite residual bound.
\end{theorem}

\subsubsection*{Proof of Theorem \ref{th-fatab}}
\begin{proof}
	Let $\vK$ be a finite set of finite algebras of finite type, and let $\vv P = \PR(\vv K)$ and $\vv Q = \QQ(\vv K)$. Clearly, $\vv P \subseteq \vv Q$, and it remains to prove that $\vv Q \subseteq \vv P$.
	
	Suppose that $n$ is the maximal cardinality of algebras from $\vv K$. Then, since $\vv Q = \QQ(\vv K)$, every $\vv Q$-irreducible algebra belongs to $\II\SU\PP_u(\vv K)$, and hence every $\vv Q$-irreducible algebra has cardinality at most $n$ and is therefore finite.
	
	By Proposition~\ref{pr-qirr}.b, every finite $\vv Q$-irreducible algebra is in $\II\SU(\vv K)$. Let $\vv K'$ be the class of all $\vv Q$-irreducible algebras, and observe that $\vv K' \subseteq \PR(\vv K)$.
	
	By Proposition~\ref{pr-qirr}.a, every algebra in $\vv Q$ is a subdirect product of $\vv Q$-irreducible algebras; that is,
	\[
	\vv Q \subseteq \II\SU\PP(\vv K').
	\]
	Therefore,
	\[
	\vv Q \subseteq \II\SU\PP(\vv K') = \PR(\vv K') \subseteq \PR(\vv K),
	\]
	and hence $\vv Q = \PR(\vv K)$.
\end{proof}

\begin{remark}
	Theorem~\ref{th-tabular1} is an algebraic generalization of a theorem from \cite{LosSuszkoRemarks1958}, which states that every consequence relation determined by a finite matrix is finitary. This also follows from \cite[Theorem~2.3.1]{GorbunovBookE}.
\end{remark}

\begin{remark}
	If $\vv K$ is an arbitrary class of finite algebras, the prevariety $\PR(\vv K)$ need not be a quasivariety. For example, the prevariety of Heyting algebras generated by the class $\vK$ of all finite linearly ordered Heyting algebras is not a quasivariety. The class $\vK$ generates the variety $\vv{LC}$, and it is known that $\vv{LC}$ contains linearly ordered Heyting algebras of arbitrary cardinality.
	
	Thus, the Heyting algebra $\alg{C}_{\omega +2}$ of order type $\omega +2$ belongs to $\vv{LC}$. Note that $\alg{C}_{\omega +2}$ is SI, and for its characteristic $\kk$-quasiequation $\csq$ we have $\alg{C}_{\omega+2} \not\models \csq$ (cf.\ Example~\ref{ex-chains}), whereas, by Proposition~\ref{pr-char}, $\vv K \models \csq$, since all algebras in $\vv K$ are finite and $\alg{C}_{\omega+2}$ is infinite and does not embed into any algebra in $\vv K$. On the other hand, $\alg{C}_{\omega +2}$ is locally embeddable in $\vK$, and hence $\alg{C}_{\omega+2} \in \QQ(\vK)$. This shows that $\PR(\vv K) \subsetneq \QQ(\vv K)$.
\end{remark}

\subsubsection*{Proof of Theorem \ref{th-fmp}}

We start with the following simple observation.

\begin{prop}\label{pr-fainfsi}
	If the variety $\vv V$ is generated by its finite algebras, then $\alg F_\vv V(\omega)$ has no infinite SI subalgebras.
\end{prop}

\begin{proof}
	For a contradiction, suppose that $\vv V$ is generated by finite algebras and that $\alg A$ is an infinite SI subalgebra of $\alg F_\vv V(\omega)$. Then there exist two distinct elements $a,b \in \alg A \subseteq \alg F_\vv V(\omega)$ that belong to the monolith of $\alg A$.
	
	Since $\vv V$ is generated by finite algebras, for any two distinct elements $a,b \in \alg F_\vv V(\omega)$, there exists a homomorphism $f \colon \alg F_\vv V(\omega) \to \alg B$ into a finite algebra $\alg B$ such that $f(a) \neq f(b)$. The restriction $f|_{\alg A}$ of $f$ to $\alg A$ is a homomorphism from $\alg A$ into $\alg B$, and $f|_{\alg A}(a) \neq f|_{\alg A}(b)$. Since $\alg A$ is infinite and $\alg B$ is finite, the homomorphism $f|_{\alg A}$ is not injective, which contradicts the fact that $a$ and $b$ lie in the monolith of $\alg A$.
\end{proof}


\begin{cor}\label{cor-FASScore}
	If a variety $\vv V$ of finite type with the CEP is generated by finite algebras, then $\vV^{\omega^+}$ contains no infinite SI algebras. 
\end{cor}

\begin{proof}
	For a contradiction, assume that $\alg{A} \in \vV^{\omega^+} = \QQ_{\omega^+}(\FrV)$ is an infinite SI algebra. Then, by Proposition~\ref{pr-CEPirr}, $\alg A$ contains a countable SI subalgebra $\alg{B}$, and $\alg B \in \QQ_{\omega^+}(\FrV)$. Let $\csq$ be the characteristic $\omega^+$-quasiequation of $\alg{B}$. Then $\alg{B} \not\models \csq$, and since $\alg B \in \QQ_{\omega^+}(\FrV)$, it follows that $\FrV \not\models \csq$. Thus, by Proposition~\ref{pr-char}, $\alg B$ is embeddable into $\FrV$, contradicting the assumption.
\end{proof}

Now we are in a position to prove Theorem~\ref{th-fmp}.

\begin{proof} 
	Suppose that $\vv V$ is an SCpl congruence meet-semidistributive variety of finite type generated by finite algebras.
	
	If $\vv V$ has a finite residual bound, then, since its type is finite, $\vv V$ has only finitely many finite SI algebras up to isomorphism; hence it is finitely generated. Since, by assumption, $\vV$ is SCpl, it follows from Corollary~\ref{cor-tab} that it is SSCpl.
	
	Conversely, suppose that $\vv V$ is an SCpl congruence meet-semidistributive variety of finite type that does not have a finite residual bound. Then SCpl implies that $\vv V = \vV^\omega$, and therefore all SI algebras in $\vv V$ belong to $\vV^\omega$. On the other hand, since $\vv V$ is congruence meet-semidistributive, of finite type, and does not have a finite residual bound, by Corollary~\ref{cor-KW}, $\vv V$ contains an infinite SI algebra $\alg A$. Hence $\alg A \in \vv V = \vV^\omega$, but by Corollary~\ref{cor-FASScore}, $\alg A \notin \vV^{\omega^+}$. Therefore, $\VP \subseteq \vV^{\omega^+} \subsetneq \vv V$, and $\vv V$ is not SSCpl. This completes the proof of Theorem~\ref{th-fmp}.
\end{proof}

Recall from \cite{BlkPgz1} that varieties with EDPC are congruence distributive and have the CEP. Hence, we have the following.

\begin{cor}\label{cor-Wradm}
	Suppose that $\vv V$ is a variety of finite type with EDPC generated by finite algebras. Then $\WsV(\omega^+)$ is admissible for $\vv V$. 
\end{cor}

\begin{proof}
	By Corollary~\ref{cor-FASScore}, $\vV^{\omega^+}$ contains no infinite SI algebras. Hence, $\vV^{\omega^+} \models \WsV(\omega^+)$, and since $\FrV \in \vV^{\omega^+}$, we have $\FrV \models \WsV(\omega^+)$; that is, $\WsV(\omega^+)$ is admissible.
\end{proof}

Thus, for all non-finitely generated varieties of finite type with EDPC that are generated by finite algebras, $\WsV(\omega^+)$ is an admissible but not derivable $\omega^+$-quasiequation.

\begin{cor}\label{cor-EDPC}
	None of the non-finitely generated varieties of finite type with EDPC generated by finite algebras are $\omega^+$-structurally complete; hence, none of them are SSCpl. 
\end{cor}

\begin{remark}
	Theorem~\ref{th-fmp} generalizes \cite[Corollary~1]{RautenbergNote1985}.
\end{remark}

\begin{remark}
	Observe that, in general, subquasivarieties of a variety generated by finite algebras need not themselves be generated by finite algebras, even if the subquasivariety is SCpl. Indeed, the variety of all Heyting algebras is generated by finite algebras, whereas its structural core, the quasivariety $\QQ(\alg{F}_{\HA}(\omega))$, is not (cf.\ \cite[Corollary~6.3.11]{RybakovBook}).
\end{remark}

\subsubsection*{Some applications} 

Corollary~\ref{cor-EDPC} has several immediate consequences.

There is a well-known correspondence between intermediate propositional logics and varieties of Heyting algebras, and between normal extensions of the logic $S4$ and varieties of interior algebras; these correspondences preserve structural completeness and strong structural completeness. Recall that the variety $\HA$ and the variety of interior algebras have EDPC.

\begin{example}
	It is known that Dummett logic $LC$ is structurally complete, and so is the corresponding variety $\vv{LC}$ of Heyting algebras generated by linearly ordered algebras (see, e.g., \cite{ChagrovZakh}). The variety $\vv{LC}$ is locally finite and contains infinite SI algebras; for example, the Heyting algebra of order type $\omega+2$. Thus, by Corollary~\ref{cor-EDPC}, variety $\vv{LC}$, as well as logic $LC$, is not strongly structurally complete. Moreover, $\vv{LC}$ is not even $\omega^+$-structurally complete.
\end{example}

\begin{example}
	It was observed in \cite{PrucnalStructural1976} that Medvedev logic $ML$, and consequently the corresponding variety of Heyting algebras $\vv{ML}$, is structurally complete, and $\vv{ML}$ is generated by finite algebras (see, e.g., \cite{ChagrovZakh}). By Corollary~\ref{cor-EDPC}, variety $\vv{ML}$, as well as logic $ML$, is not strongly structurally complete, and is not even $\omega^+$-structurally complete.
\end{example}

\begin{example}
	The variety of Brouwerian algebras (implicative semilattices) is locally finite (see, e.g., \cite{KohlerBrouwerian1981}); hence any variety of Brouwerian algebras containing an infinite SI algebra is neither strongly structurally complete nor $\omega^+$-structurally complete.
\end{example}

\begin{example}
	By a similar argument, logic $S4.3$ is not strongly structurally complete (infinitary structural completeness of $S4.3$ and its extensions was studied in depth in~\cite{DzikWojtylakAlmost2015}).
\end{example}

\section{Strongly primitive varieties}\label{sec-primitive}

In this section we discuss how hereditary structural completeness and hereditary strong structural completeness are related.

\begin{definition}
	Let $\vV$ be a variety and $\kk$ be an infinite cardinal. Then,
	\begin{enumerate}
		\item[] $\vv V$ is \Def{hereditarily structurally complete}, or \Def{primitive}, if $\vv V$ and all its subvarieties are SCpl;
		
		\item[] $\vv V$ is \Def{hereditarily $\kk$-strongly structurally complete}, or \Def{$\kk$-primitive}, if $\vv V$ and all its subvarieties are $\kk$-structurally complete;
		
		\item[] $\vv V$ is \Def{hereditarily strongly structurally complete}, or \Def{strongly primitive}, if $\vv V$ and all its subvarieties are SSCpl.
	\end{enumerate}
\end{definition}

Theorem~\ref{th-fatab} yields the following.

\begin{cor}\label{cor-superpr}
	If $\vv V$ is a congruence distributive HSCpl variety of finite type, every subvariety of which is generated by finite algebras (in particular, if $\vv V$ is locally finite), then the finitely generated subvarieties of $\vv V$, and only those, are HSSCpl.
\end{cor}

It follows from \cite{Citkin1978} that all primitive varieties of Heyting algebras are locally finite. Hence, by Corollary~\ref{cor-superpr}, only tabular primitive varieties are strongly primitive. In particular, the following holds.

\begin{theorem}[{\cite[Section~4]{Citkin1978}}]
	A variety of Heyting algebras is strongly primitive if and only if it is primitive and omits the Heyting algebra of order type $\omega+2$.
\end{theorem}

The $\{\to\}$, $\{\to, \land\}$, and $\{\to,\land,\bot\}$-fragments of intuitionistic propositional logic are HSCpl and are known to be locally tabular. Thus, the corresponding varieties are primitive and locally finite. By Corollary~\ref{cor-superpr}, all tabular subvarieties of these varieties, and only those, are strongly primitive. 

\begin{remark}
	Hereditary structural completeness in the infinitary sense for the $\{\to\}$ and $\{\to, \land\}$-fragments was observed and studied in \cite{PrucnalStructural1983,PrucnalStructural1985,WojtylakSyntactical1990, WojtylakStructural1991}. 
\end{remark}

\begin{example}
	Denote by $\vv {K4}$ the variety of modal algebras corresponding to the logic $K4$. 
	Recall (see, e.g., \cite[Section~5]{RybakovBook}) that all primitive subvarieties of $\vv {K4}$ are locally finite. By Corollary~\ref{cor-superpr}, the tabular subvarieties of $\vv {K4}$, and only those, are strongly primitive.
\end{example}

\subsubsection*{Criterion of strong primitivity}
Similarly to the case of primitivity for locally finite quasivarieties, we have the following.

\begin{theorem}
	A variety $\vv V$ of finite type with the CEP is strongly primitive if and only if all its countable SI algebras are weakly $\vv V$-projective.
\end{theorem}

\begin{proof}
	Let $\vv V$ be a variety of finite type.
	
	From left to right. Suppose that $\vv V$ is strongly primitive and that $\alg A \in \vv V$ is a countable SI algebra. We prove the contrapositive: if $\alg A$ is not weakly $\vv V$-projective, then there exists a subquasivariety $\vv V' \subseteq \vv V$ that is not SSCpl.
	
	Assume that $\alg B \in \vv V$ is such that $\alg A \in \HH(\alg B)$ whereas $\alg A \notin \II\SU(\alg B)$. We show that $\VV(\alg B)$ is not SSCpl.
	
	Since $\alg A \notin \II\SU(\alg B)$, by Proposition~\ref{pr-char}, $\alg B \models \csq_\alg A$. Moreover, since $\alg B$ generates $\VV(\alg B)$, by Proposition~\ref{pr-adm}, $\csq_\alg A$ is admissible for $\VV(\alg B)$. However, $\alg A \in \HH(\alg B) \subseteq \VV(\alg B)$ and $\alg A \not\models \csq_\alg A$, which entails that $\VV(\alg B)$ is not SSCpl.
	
	From right to left. Suppose that all countable SI algebras in $\vv V$ are weakly $\vv V$-projective. Again, we prove the contrapositive: let $\vv V' \subseteq \vv V$ and suppose that $\vv V'$ is not SSCpl. Then $\PR(\alg F_{\vv V'}(\omega)) \subsetneq \vv V'$. Observe that $\vv V' \setminus \PR(\alg F_{\vv V'}(\omega))$ must contain an SI algebra; otherwise, since every algebra in $\vv V'$ is a subdirect product of its SI algebras, and prevarieties are closed under subdirect products, we would have $\vv V' = \PR(\alg F_{\vv V'}(\omega))$, contrary to our assumption.
	
	Suppose that $\alg A \in \vv V' \setminus \PR(\alg F_{\vv V'}(\omega))$ is an SI algebra. Then there is a $\kk$-quasiequation $\sq$ such that $\PR(\alg F_{\vv V'}(\omega)) \models \sq$ whereas $\alg A \not\models \sq$. By Corollary~\ref{cor-CEPref}, $\alg A$ contains a countable SI subalgebra $\alg B$ with $\alg B \not\models \sq$. Because $\alg B$ is countable, it is a homomorphic image of $\alg F_{\vv V'}(\omega)$, but $\alg B \not\models \sq$ and hence $\alg B \notin \II\SU(\alg F_{\vv V'}(\omega))$, that is, $\alg B$ is not weakly $\vv V$-projective.
\end{proof}

\section{Conclusion}

In conclusion, we formulate some open problems.

\begin{enumerate}
	\item Is it true that for every variety $\vv V$ of finite type generated by finite algebras with EDPC, the strong structural core can be defined relative to the structural core by $\WsV(\omega^+)$?
	
	\item Is it true that, by adding the bounding rule to the Visser rules, one obtains an axiomatization of the strongly structurally complete consequence relation for IPC?
	
	\item Is it true that the bounding rule yields an axiomatization of the strongly structurally complete consequence relation for $LC$ and $ML$?
\end{enumerate}

\bibliographystyle{eptcs}

\end{document}